\documentclass[11pt]{article}
\usepackage[english]{babel}
\usepackage{amssymb,amsmath,amsthm}
\newtheorem{theorem}{Theorem}[section]
\newtheorem{lemma}[theorem]{Lemma}

\newtheorem{question}[theorem]{Question}

{\theoremstyle{definition}}
{\theoremstyle{definition}}
{\theoremstyle{definition}}
{\theoremstyle{definition}}

\numberwithin{equation}{section}

\def\C{{\mathbb C}}
\def\P{{\mathbb P}}
\def\N{{\mathbb N}}
\def\Z{{\mathbb Z}}
\def\R{{\mathbb R}}
\def\H{{\mathcal H}}
\def\epsilon{\varepsilon}
\def\phi{\varphi}
\def\leq{\leqslant}
\def\geq{\geqslant}

\def\deg{\hbox{\tt deg}\,}
\def\bin#1#2{\left({{#1}\atop {#2}}\right)}

\title{On the set of hypercyclic vectors for the differentiation operator}

\author{Stanislav Shkarin}

\date{}

\begin{document}

\maketitle

\begin{abstract} Let $D$ be the differentiation operator $Df=f'$
acting on the Fr\'echet space $\H$ of all entire functions in one
variable with the standard (compact-open) topology. It is known
since 1950's that the set $H(D)$ of hypercyclic vectors for the
operator $D$ is non-empty. We treat two questions raised by Aron,
Conejero, Peris and Seoane-Sep\'ulveda whether the set $H(D)$
contains (up to the zero function) a non-trivial subalgebra of $\H$
or an infinite dimensional closed linear subspace of $\H$. In the
present article both questions are answered affirmatively.
\end{abstract}

\small \noindent{\bf MSC:} \ \ 47A16, 37A25

\noindent{\bf Keywords:} \ \ Hypercyclic operators, Entire
functions, Hypercyclic vectors

\normalsize

\section{Introduction \label{s1}}\rm

As usual, $\C$ is the field of complex numbers, $\Z_+$ is the set of
non-negative integers and $\N$ is the set of positive integers. Let
$X$ be a topological vector space and $T$ be a continuous linear
operator acting on $X$. Recall that $x\in X$ is called a {\it
hypercyclic vector} for $T$ if the orbit $\{T^nx:n\in\Z_+\}$ is
dense in $X$. By $H(T)$ we denote the set of hypercyclic vectors for
$T$. The operator $T$ is called {\it hypercyclic} if it has a
hypercyclic vector. For more information on hypercyclic operators
see the surveys \cite{ge1,ge2} and references therein. We would just
like to mention that the set $H(T)$ for any hypercyclic operator $T$
contains all non-zero vectors from a dense linear subspace of $X$.
It follows from the fact due to Bourdon \cite{bourd} (see also
\cite{ansa}) that if $x\in H(T)$, then $p(T)x\in H(T)$ for any
non-zero polynomial $p$. The question whether $H(T)$ for a given
operator $T$ must contain all non-zero vectors from a closed
infinite dimensional subspace of $X$ was studied by several authors.
See \cite{alf1,alf2} for sufficient conditions in terms of the
spectrum of $T$ for $H(T)$ to contain all non-zero vectors from an
infinite dimensional closed linear linear subspace of $X$ in the
case when $X$ is a complex Banach space.

By $\H$ we denote the space of all entire functions $f:\C\to\C$ with
the topology of uniform convergence on compact sets. It is
well-known that $\H$ is a Fr\'echet space. That is, $\H$ is complete
metrizable locally convex space, whose topology is defined by the
increasing sequence of norms $f\mapsto \max\{|f(z)|:|z|\leq n\}$ for
$n\in\N$. The differentiation operator
\begin{equation*}
D:\H\to\H,\quad Df=f'
\end{equation*}
is a continuous linear operator on $\H$. Due to MacLane \cite{mac},
$D$ is hypercyclic. It is well-known \cite{ge1} that the set of
hypercyclic vectors of any hypercyclic operator on a separable
metrizable topological vector space is a dense $G_\delta$-set. Hence
$H(D)$  is a dense $G_\delta$-set in $\H$. We deal with two problems
raised by Aron, Conejero, Peris and Seoane-Sep\'ulveda in
\cite{aron1}. It is worth mentioning that $\H$ is an algebra with
respect to pointwise multiplication.

\begin{question}\label{ar1} Does $H(D)$ contain all non-zero vectors
from a closed infinite dimensional linear subspace of $\H$?
\end{question}

\begin{question}\label{ar2} Does $H(D)$ contain all non-constant
functions from a non-trivial subalgebra of $\H$? In other words,
does there exist $f\in\H$ such that $p\circ f\in H(D)$ for any
non-constant polynomial $p$?
\end{question}

Note that the analog of the last question for the translation
operator $Tf(z)=f(z-1)$ on $\H$ has been answered negatively by the
same set of authors \cite{aron2}. Namely, they have shown that for
any $f\in\H$ and any $k\geq 2$, $f^k\notin H(T)$. In \cite{aron1} it
is also shown that the set $\{f\in\H:f^n\in H(D)\ \ \text{for any}\
n\in\N\}$ is a dense $G_\delta$-set in $\H$, thus providing an
evidence that the answer to Question~\ref{ar2} could be affirmative.
In the present paper both above questions are answered affirmatively
and constructively. It is worth noting that Question~\ref{ar2} was
recently independently answered by Bayart and Matheron by means of
applying the Baire theorem. Their proof will soon appear in the book
\cite{bama}.

\begin{theorem}\label{main1} There is a closed infinite
dimensional subspace $L$ of $\H$ such that $L\setminus\{0\}\subset
H(D)$.
\end{theorem}

\begin{theorem}\label{main2} There exists $f\in\H$ such that $p\circ f\in H(D)$ for
any non-constant polynomial $p$.
\end{theorem}

\section{Preliminaries}

Before proving Theorems~\ref{main1} and~\ref{main2}, we would like
to introduce some notation and mention few elementary facts.
Throughout the paper $\P$ stands for the space $\C[z]$ of all
complex polynomials in one variable. Clearly $\P$ is a dense linear
subspace of $\H$. Let
$$
\P_0=\{0\}\ \ \text{and}\ \ \P_k=\{p\in\P:\deg p<k\}\ \ \text{for
$k\in\N$}.
$$
Obviously $\P_k$ is a $k$-dimensional linear subspace of $\P$. If
$k\in\N$ and $c>0$, we denote
$$
\P_{k,c}=\biggl\{p(z)=\sum_{j=0}^{k-1}c_jz^j:|c_j|\leq c\ \
\text{for}\ \ 0\leq j\leq k-1\}.
$$
Since we are going to deal with the Taylor series expansion of
functions $f\in\H$ rather than their values, we consider a sequence
of norms defining the topology of $\H$ different from the one
mentioned in the introduction. Namely for $a\in\N$ and $f\in\H$, we
write
\begin{equation}\label{norm}
\|f\|_a=\sum_{n=0}^\infty |f_n|a^n,\ \ \text{where $f\in\H$,}\ \
f(z)=\sum_{n=0}^\infty f_nz^n.
\end{equation}
It is easy to see that the above sequence of norms is increasing and
defines the original topology on $\H$. Moreover, each of these norms
is submultiplicative:
\begin{equation}\label{norm1}
\text{$\|f\|_a\leq\|f\|_b$ and $\|fg\|_a\leq \|f\|_a\|g\|_a$
whenever $f,g\in\H$, $a,b\in\N$, $a\leq b$.}
\end{equation}

Observe that $D(\P_k)\subseteq \P_{k-1}$ for any $k\in\N$. In
particular, $D^n(\P_k)=\{0\}$ if $n\geq k$. Moreover, $\|Dp\|_a\leq
\frac{k-1}{a}\|p\|_a$ for each $k,a\in\N$ and any $p\in\P_k$.
Iterating this estimate, we obtain
\begin{equation}\label{normd2}
\|D^np\|_a\leq \frac{|(k-n)\dots(k-1)|}{a^n}\|p\|_a\leq
(k/a)^n\|p\|_a\ \ \text{for any $k,n,a\in\N$ and $p\in\P_k$}.
\end{equation}
We also consider the Volterra operator $V:\H\to\H$, $Vf(z)=\int_0^z
f(t)\,dt$. It is easy to see that
\begin{equation}\label{vo}
Vf(z)=\sum_{n=1}^\infty \frac{f_{n-1}}{n}z^n,\ \ \text{where
$f\in\H$,}\ \ f(z)=\sum_{n=0}^\infty f_nz^n
\end{equation}
and that $V$ is a right inverse of $D$. In particular,
\begin{equation}\label{ri}
D^nV^n=I\ \ \text{for any}\ \ n\in\N.
\end{equation}
Using (\ref{vo}), one can easily verify that
\begin{equation}\label{normv1}
\|V^nf\|_a\leq \frac{a^n}{n!}\|f\|_a\ \ \text{for any $n,a\in\N$ and
any $f\in\H$.}
\end{equation}
For $f\in\H$, the {\it support} of $f$ is the set
$\{n\in\Z_+:f^{(n)}(0)\neq 0\}$. Obviously $D$ shifts the supports
to the left and $V$ shifts them to the right. That is, if $A$ is the
support of $f$, then $A+1=\{n+1:n\in A\}$ is the support of $Vf$ and
$(A-1)\cap \Z_+$ is the support of $Df$.

In the proof of Theorems~\ref{main1} and~\ref{main2} we use a
sequence in $\N\times\P$ with specific properties.

\begin{lemma}\label{NtP} There exists a sequence
$\{(d_k,p_k)\}_{k\in\N}$ of elements of $\N\times\P$ such that
\begin{itemize}\itemsep=-3pt
\item[\rm(\ref{NtP}.1)]$d_k\leq k$ and $p_k\in\P_{k,k}$ for each
$k\in\N;$
\item[\rm(\ref{NtP}.2)]for any $d\in\N$, the set $\{p_k:d_k=d\}$ is
dense in $\H$.
\end{itemize}
\end{lemma}

\begin{proof} Since $\P$ is dense in $\H$ and $\P$ is the union of $\P_{k,k}$
for $k\in\N$, we can pick a sequence $\{s_j\}_{j\in\N}$ in $\P$ such
that $\{s_j:j\in\N\}$ is dense in $\H$ and $s_j\in\P_{j,j}$ for any
$j\in\N$. It is well-known and easy to see that there is a bijection
$\phi:\N\times\N\to \N$ such that $\max\{m,j\}\leq \phi(m,j)$ for
any $m,j\in\N$. We define a sequence $\{(d_k,p_k)\}_{k\in\N}$ of
elements of $\N\times\P$ by the formula
$$
(d_k,p_k)=(m,s_j)\ \ \text{if}\ \ \phi(m,j)=k.
$$
The condition $\max\{m,j\}\leq \phi(m,j)$ and the obvious inclusion
$P_{j,j}\subseteq \P_{k,k}$ for $j\leq k$ imply that (\ref{NtP}.1)
is satisfied. Next, let $d\in\N$. From the definition of $(d_k,p_k)$
and bijectivity of $\phi$ it follows that
$\{p_k:d_k=d\}=\{s_j:j\in\N\}$. Hence (\ref{NtP}.2) is also
satisfied.
\end{proof}

\section{Proof of Theorem~\ref{main1}}

Let $\{(d_k,p_k)\}_{k\in\N}$ be the sequence of elements of
$\N\times\P$ provided by Lemma~\ref{NtP}. For each $d\in\N$ let
$B_d=\{k\in\N:d_k=d\}$. By Lemma~\ref{NtP}, $B_d$ are infinite
disjoint subsets of $\N$, whose union is $\N$. Let $m_d=\min B_d$
and $B'_d=B_d\setminus\{m_d\}$. By (\ref{NtP}.1), $m_d\geq d$. We
also need a sequence increasing fast enough. Namely, pick
$\beta:\N\to\N$ such that
\begin{equation}\label{beta}
\text{$\beta(k+1)>\beta(k)+k$ and $\beta(k+1)^{\beta(k)}\leq
2^{\beta(k+1)}$ for any $k\in\N$}.
\end{equation}
For each $d\in\N$, we consider the series
$$
f_d=g_d+\sum_{k\in B'_d} V^{\beta(k)}p_k,\ \ \text{where}\ \
g_d(z)=z^{\beta(m_d)}.
$$
Let $a\in\N$. Since $p_k\in\P_{k,k}$, we have $\|p_k\|_a\leq
k^2a^k$. Thus using (\ref{normv1}), we obtain
$$
\sum_{k\in B'_d} \|V^{\beta(k)}p_{k}\|_a\leq \sum_{k=1}^\infty
\|V^{\beta(k)}p_k\|_a\leq \sum_{k=1}^\infty \frac{k^2a^k
a^{\beta(k)}}{\beta(k)!}<\infty.
$$
Hence the series defining $f_d$ converges absolutely and therefore
$f_d\in\H$ for $d\in\N$. The inclusions $p_k\in\P_k$ and the
inequality $\beta(k+1)>\beta(k)+k$ imply that the supports of $g_d$
and $V^{\beta(k)}p_{k}$ are pairwise disjoint. Hence the supports of
$f_d$ are pairwise disjoint. It is easy to verify that each sequence
of non-zero functions in $\H$ with pairwise disjoint supports is a
Schauder basic sequence. Hence $\{f_d\}_{d\in\N}$ is a Schauder
basic sequence in $\H$ and the closed linear span $L$ of
$\{f_d:d\in\N\}$ consists of the sums of convergent series of the
shape $\sum\limits_{d=1}^\infty c_df_d$ with $c_d\in\C$. In order to
prove Theorem~\ref{main1}, it is enough to demonstrate that
$L\setminus\{0\}\subseteq H(D)$.

Let $f\in L\setminus\{0\}$. Then $f$ is the sum of a convergent
series $\sum\limits_{d=1}^\infty c_df_d$ with $c_d\in\C$ being not
all zero. Since a non-zero scalar multiple of a hypercyclic vector
is hypercyclic, we, multiplying $f$ by a non-zero constant, can
assume that there is $b\in\N$ such that $c_b=1$. Considering the
natural projection onto the subspace of $\H$ of functions whose
support is contained in $\{\beta(m_d):d\in\N\}$, we see that the
series $\sum\limits_{d=1}^\infty c_dg_d$ converges in $\H$. Since
$g_d(z)=z^{\beta(m_d)}$, it follows that $|c_d|^{1/\beta(m_d)}\to
0$. In order to verify that $f\in H(D)$ it suffices to demonstrate
that
\begin{equation}\label{clo1}
\text{$D^{\beta(k)}f-p_k\to 0$ in $\H$ as $k\to\infty$, $k\in B_b$}.
\end{equation}
Indeed, by Lemma~\ref{NtP}, $\{p_k:k\in B_b\}$ is dense in $\H$.
Then (\ref{clo1}) implies that $\{D^{\beta(k)}f:k\in B_b\}$ is dense
in $\H$ and therefore $f\in H(D)$.

It remains to prove (\ref{clo1}). Let $C=\{m_d:d\in\N\}$. Using
definitions of $f$ and $f_d$ and the condition
$|c_d|^{1/\beta(m_d)}\to 0$ it is easy to see that
$$
f=\sum_{d\in\N}c_df_d=g+h,\ \ \text{where}\ \ g=\sum_{d\in\N}c_dg_d\
\ \text{and}\ \ h=\sum_{k\in\N\setminus C} c_{d_k}V^{\beta(k)}p_k,
$$
where the series defining $h$ and $g$ are absolutely convergent in
$\H$. Let $a\in\N$ and $k\in B'_b$. Since $D^n(\P_j)=\{0\}$ for
$n\geq j$, we, using the above display together with (\ref{ri}),
obtain
\begin{equation}\label{sum}
D^{\beta(k)}f=p_k+q_k+h_k,\ \ \text{where}\ \
q_k=\sum_{m_d>k}c_dD^{\beta(k)}g_d\ \ \text{and}\ \
h_k=\sum_{n\notin C,\ n>k} c_{d_n}V^{\beta(n)-\beta(k)}p_n.
\end{equation}
Condition $|c_d|^{1/\beta(m_d)}\to 0$ implies that there is
$c=c(a)>0$ such that $|c_d|\leq c(4a)^{-\beta(m_d)}$ for any
$d\in\N$. As we have already mentioned, $\|p_n\|_a\leq n^2a^n$ for
any $n\in\N$. By (\ref{beta}) $\beta(n)-\beta(k)\geq n$ for any
$n>k$. Hence, using (\ref{normv1}) and the inequality $|c_d|\leq c$,
we have
$$
\|h_k\|_a\leq \sum_{n=k+1}^\infty
\frac{cn^2a^{\beta(n)-\beta(k)+n}}{(\beta(n)-\beta(k))!}\leq
\!\!\!\sum_{n=k+1}^\infty
\frac{c(\beta(n)-\beta(k))^2a^{2(\beta(n)-\beta(k))}}{(\beta(n)-\beta(k))!}
\leq \!\!\!\sum_{m=k+1}^\infty \frac{cm^2a^{2m}}{m!}\to 0
$$
as $k\to\infty$.  Next, we estimate $\|q_k\|_a$. Using the
inequality $|c_d|\leq c(4a)^{-\beta(m_d)}$, we obtain
$$
\|q_k\|_a\leq\!\!
\sum_{m_d>k}\!\!c(4a)^{-\beta(m_d)}\|D^{\beta(k)}g_d\|_a\leq c\!\!
\sum_{m_d>k}\!\!
(4a)^{-\beta(m_d)}\beta(m_d)^{\beta(k)}a^{\beta(m_d)}=
c\!\!\sum_{m_d>k}\! 4^{-\beta(m_d)}\beta(m_d)^{\beta(k)},
$$
where the last inequality in the above display follows from the
equalities $g_d(z)=z^{\beta(m_d)}$ and (\ref{normd2}). According to
(\ref{beta}), $\beta(m_d)^{\beta(k)}\leq 2^{\beta(m_d)}$ whenever
$m_d>k$. Substituting these inequalities into the above display we
arrive to
$$
\|q_k\|_a\leq c \sum_{m_d>k} 2^{-\beta(m_d)}.
$$
It follows that $\|q_k\|_a\to 0$ as $k\to\infty$. As we already
know, $\|h_k\|_a\to 0$ as $k\to\infty$. Since $a\in\N$  is
arbitrary, $q_k\to 0$ and $h_k\to 0$ in $\H$ as $k\to\infty$. By
(\ref{sum}), $D^{\beta(k)}f-p_k\to 0$ in $\H$ as $k\to\infty$, $k\in
B_b$. The proof of (\ref{clo1}) and of Theorem~\ref{main1} is now
complete.

\section{Proof of Theorem~\ref{main2}}

The main building blocks of our construction of a generator of an
algebra contained in $H(D)$ are polynomials of the shape
\begin{equation}\label{ra}
r_{\alpha,n}(z)=\frac{z^n}{n^n}+q_{\alpha,n}(z),\ \ \text{where}\ \
q_{\alpha,n}(z)=n^{(d-1)n}\frac{V^{n^2+(d-1)n}p(z)}{dz^{(d-1)n}},\
n>1,\ \alpha=(d,p)\in \N\times\P.
\end{equation}

Direct calculations show that
\begin{equation}\label{ra1}
\text{if}\ p(z)=\sum\limits_{j=0}^{k-1}c_jz^j,\ \ \text{then}\ \
q_{\alpha,n}(z)=\frac{n^{(d-1)n}}{d}\sum_{j=0}^{k-1}\frac{j!c_jz^{j+n^2}}{(j+(d-1)n+n^2)!}.
\end{equation}
The next lemma is the reason for the choice of $r_{\alpha,n}$.

\begin{lemma}\label{est1} Let $\alpha=(d,p)\in \N\times\P$ and $a\in\N$.
Then
\begin{align}\label{ra2}
&\lim_{n\to\infty}\|r_{\alpha,n}\|_a=0,
\\ \label{ra3}
& \text{for any $\nu,b\in\N$ and $h\in\P$,}\ \
\lim_{n\to\infty}\|D^{\nu}(hr_{\alpha,n}^b)\|_a=0,
\\ \label{ra4}
&\text{for any $h\in\P$ and $b\in\N$, $1\leq b<d$,}\ \
\lim_{n\to\infty}\|D^{(d-1)n+n^2}(hr_{\alpha,n}^b)\|_a=0,
\\ \label{ra5}
&\lim_{n\to\infty}\|p-D^{(d-1)n+n^2}(r_{\alpha,n}^d)\|_a=0.
\end{align}
\end{lemma}

\begin{proof} Pick $k\in\N$ such that $p\in\P_{k,k}$. Then
$p(z)=\sum\limits_{j=0}^{k-1}c_jz^j$ with $|c_j|\leq k$ for $0\leq
j\leq k-1$. According to (\ref{ra1}) and (\ref{norm}),
$$
\|q_{\alpha,n}\|_a=
\sum_{j=0}^{k-1}\frac{n^{(d-1)n}j!|c_j|a^{j+n^2}}{d(j+(d-1)n+n^2)!}\leq
\sum_{j=0}^{k-1}\frac{k!n^{(d-1)n}a^{k+n^2}}{(j+(d-1)n+n^2)!}\leq
\frac{2k!n^{(d-1)n}a^{k+n^2}}{(n^2)!},
$$
where we have used the inequalities $|c_j|j!\leq k!$ for $0\leq
j\leq k-1$. Using the Stirling formula, we easily see that
$(n^2)!\geq 2(n^2/e)^{n^2}$. From this inequality and the above
display we see that $\|q_{\alpha,n}\|_a\leq
k!a^kn^{(d-1)n}(ea)^{n^2}n^{-2n^2}$ for $n>1$. Hence there exists
$c=c(a,d,p)>0$ such that
\begin{equation}\label{qan}
\|q_{\alpha,n}\|_a\leq (cn)^{-2n^2}\ \ \text{for any $n>1$}.
\end{equation}
By (\ref{ra}), $r_{\alpha,n}(z)=(z/n)^n+q_{\alpha,n}(z)$. Hence
$\|r_{\alpha,n}\|_a\leq (a/n)^n+\|q_{\alpha,n}(z)\|_a$ and therefore
(\ref{ra2}) follows from (\ref{qan}).

According to (\ref{qan}), we can pick $c_1=c_1(a,p,d)>1$ such that
$\|q_{\alpha,n}\|_a\leq (c_1-1)(a/n)^n$ for any $n>1$. Then
$\|r_{\alpha,n}\|_a\leq (a/n)^n+\|q_{\alpha,n}(z)\|_a\leq
c_1(a/n)^n$ for any $n>1$. Let $h\in\P$ and $b,\nu\in\N$. Using
submultiplicativity of the norm $\|\cdot\|_a$, we obtain
$$
\|hr_\alpha^b\|_a\leq c_1^b\|h\|_a (a/n)^{bn}\ \ \text{for any
$n>1$}.
$$
Next, pick $\mu\in\N$ such that $h\in\P_\mu$. By (\ref{ra}),
$hr_\alpha^b\in\P_{bn^2+bk+\mu}$. Thus, applying the above display
and the estimate (\ref{normd2}), we have
$$
\|D^\nu(hr_\alpha^b)\|_a\leq c_1^b\|h\|_a
(bn^2+bk+\mu)^\nu(a/n)^{bn}\ \ \text{for any $n>1$}.
$$
Equality (\ref{ra3}) follows immediately from the above estimate.

Finally, assume that $1\leq b\leq d$. Since
$r_{\alpha,n}(z)=(z/n)^n+q_{\alpha,n}(z)$, we have
\begin{align*}
&(hr_{\alpha,n}^b)(z)=\sum_{j=0}^b \bin{b}{j}
\frac{z^{nj}q_{\alpha,n}(z)^{b-j}h(z)}{n^{nj}}=f_n(z)+g_n(z),\ \
\text{where}
\\
&f_n(z)=\frac{z^{bn}h(z)}{n^{bn}}+\frac{bq_{\alpha,n}(z)h(z)z^{(b-1)n}}{n^{(b-1)n}}\
\ \text{and}\ \ g_n(z)=\sum_{0\leq j\leq b-2}\bin{b}{j}
\frac{z^{nj}h(z)q_{\alpha,n}(z)^{b-j}}{n^{nj}}.
\end{align*}
First, we shall estimate $\|g_n\|_a$. Using (\ref{qan}) and
submultiplicativity of $\|\cdot\|_a$, we get
$$
\|g_n\|_a\leq \|h\|_a\sum_{0\leq j\leq b-2}\bin{b}{j}
\frac{a^{nj}}{n^{nj}}(cn)^{-2(b-j)n^2}\ \ \text{for any $n>1$}.
$$
For $n\geq a$ we have $(a/n)^{nj}\leq 1$. Since $b-j$ in the above
sum is at least $2$, for $n\geq c^{-1}$ we have
$(cn)^{-2(b-j)n^2}\leq (cn)^{-4n^2}$. Hence, we can write
\begin{equation}\label{ges}
\|g_n\|_a\leq \|h\|_a\,(cn)^{-4n^2}\sum_{0\leq j\leq b-2}\bin{b}{j}
\leq 2^b\|h\|_a\,(cn)^{-4n^2}\ \ \text{for $n> \max\{a,c^{-1}\}$}.
\end{equation}
Recall that $h\in\P_\mu$ and $p\in\P_k$. Then
$g_n\in\P_{bn^2+bk+\mu}$. According to (\ref{normd2}),
$\|D^{n^2+(d-1)n}g_n\|_a\leq (bn^2+bk+\mu)^{n^2+(d-1)n}\|g_n\|_a$.
By (\ref{ges}),
$$
\|D^{n^2+(d-1)n}g_n\|_a\leq
2^b\|h\|_a\,(bn^2+bk+\mu)^{n^2+(d-1)n}(cn)^{-4n^2}\ \ \text{for $n>
\max\{a,c^{-1}\}$}.
$$
Passing to the limit as $n\to\infty$ we arrive to
\begin{equation}\label{ges1}
\lim_{n\to\infty}\|D^{n^2+(d-1)n}g_n\|_a =0.
\end{equation}
Since $\deg q_{\alpha,n}<n^2+k$ and $\deg h<\mu$, we see that $\deg
f_n< n^2+k+\mu+(b-1)n-1$. Hence $\deg f_n<n^2+(d-1)n+k+\mu-(d-b)n$.
If $b<d$ and $n\geq k+\mu$, we have $\deg f_n<n^2+(d-1)n$. Hence
$D^{n^2+(d-1)n}f_n=0$. Since $hr_{\alpha,n}^b=f_n+g_n$, we have
$D^{n^2+(d-1)n}(hr_{\alpha,n}^b)=D^{n^2+(d-1)n}g_n$ for $n\geq
k+\mu$ and (\ref{ra4}) follows from (\ref{ges1}).

Now consider the case $b=d$ and $h=1$. In this case
$$
f_n(z)=\frac{z^{dn}}{n^{dn}}+\frac{dz^{(d-1)n}q_{\alpha,n}(z)}{n^{(d-1)n}}.
$$
Since $n^2>n$, we have $n^2+(d-1)n>dn$ and therefore
$D^{n^2+(d-1)n}$ annihilates the first summand in the above display.
Since $q_{\alpha,n}(z)=r_{\alpha,n}(z)-(z/n)^n$, from (\ref{ra}) it
follows that
$$
\frac{dz^{(d-1)n}q_{\alpha,n}(z)}{n^{(d-1)n}}=(V^{n^2+(d-1)n}p)(z).
$$
According to (\ref{ri}), we obtain
$D^{n^2+(d-1)n}f_n=D^{n^2+(d-1)n}V^{n^2+(d-1)n}p=p$. Since
$r_{\alpha,n}^d=f_n+g_n$,
$p-D^{n^2+(d-1)n}(r_{\alpha,n}^d)=-D^{n^2+(d-1)n}g_n$ and
(\ref{ra5}) follows from (\ref{ges1}).
\end{proof}

We are ready to prove Theorem~\ref{main2}. Let
$\{(d_k,p_k)\}_{k\in\N}$ be the sequence in $\N\times\P$ provided by
Lemma~\ref{NtP}. Denote $\alpha_k=(d_k,p_k)$. We shall construct
inductively natural numbers $n_k\geq 2$ such that for any $k\in\N$,
\begin{itemize}\itemsep=-3pt
\item[(a1)] $n_k>n_{k-1}$ if $k\geq 2$;
\item[(a2)] $\|r_k\|_k\leq 2^{-k}$, where $r_k=r_{\alpha_k,n_k}$;
\item[(a3)] if $k\geq 2$, then $\|D^\nu (f_k^j-f_{k-1}^j)\|_k\leq 2^{-k}$
for any $j\leq k$ and $\nu\leq n^2_{k-1}+kn_{k-1}$, where
$f_a=\sum\limits_{l=1}^a r_l$;
\item[(a4)]$\|D^{\nu_k} (f_k^{j})\|_k\leq 2^{-k}$ for $1\leq
j<d_k$ and $\|p_k-D^{\nu_k} (f_k^{d_k})\|_k\leq 2^{-k}$, where
$\nu_k=n_k^2+(d_k-1)n_k$.
\end{itemize}

At step 1 we take $n_1=3$. Conditions (a1) and (a3) for $k=1$ are
trivially satisfied. According to (\ref{NtP}.1), $d_1=1$ and
$p_1=c$, where $c\in\C$ is a constant, $|c|\leq1$. By (\ref{ra}),
$r_1(z)=\frac{z^3}{27}+\frac{cz^9}{9!}$ and therefore
$\|r_1\|_1\leq\frac1{27}+\frac1{9!}<2^{-1}$ and (a2) for $k=1$ is
satisfied. Since $\nu_1=9$, $d_1=1$ and $D^9r_1=c=p_1$, (a4) for
$k=1$ is also satisfied. This provides us with the basis of
induction.

Assume now that $m\geq 2$ and $n_k$ for $1\leq k\leq m-1$ satisfying
(a1--a4) are already constructed. We have to construct $n_m$
satisfying (a1--a4) for $k=m$. We shall actually show that any
sufficiently large $n_m$ satisfies (a1--a4). For any $n\in\N$, $n>1$
consider
$$
\text{$\rho_{n}=r_{\alpha_m,n}$, $\phi_n=f_{m-1}+\rho_n$ and
$\beta_n=n^2+(d_m-1)n$}.
$$
Applying Lemma~\ref{est1}, we obtain
\begin{align}\label{e1}
&\lim_{n\to\infty}\|\rho_n\|_m=0.
\\ \label{e2}
&\lim_{n\to\infty}\|D^\nu(h\rho_n^j)\|_m=0\ \ \text{for any
$\nu,j\in\N$ and $h\in\P$.}
\\ \label{e3}
&\lim_{n\to\infty}\|p_m-D^{\beta_n}(\rho_n^{d_m})\|_m=0\ \
\text{and}\ \  \lim_{n\to\infty}\|D^{\beta_n}(h\rho_n^j)\|_m=0\ \
\text{for any $h\in\P$ and $1\leq j<d_m$.}
\end{align}
Using the binomial formula, we write
\begin{equation}\label{bi}
\phi_n^j-f_{m-1}^j=\sum_{l=0}^{j-1}\bin{j}{l} f_{m-1}^l\rho_n^{j-l}.
\end{equation}
From (\ref{bi}) and (\ref{e2}) it follows that
$\|D^\nu(\phi_n^j-f_{m-1}^j)\|_m\to 0$ as $n\to\infty$ for any
$j,\nu\in\N$. Hence
\begin{equation}\label{e4}
\lim_{n\to\infty}\max_{1\leq j\leq m\atop 1\leq\nu\leq b}
\|D^{\nu}(\phi_n^j-f_{m-1}^j)\|_m=0,\ \ \text{where
$b=n_{m-1}^2+mn_{m-1}$.}
\end{equation}
According to (\ref{bi}) with $j\leq d_m$ and (\ref{e3}),
\begin{equation}\label{e5}
\lim_{n\to\infty}\|p_m-D^{\beta_n}\phi_n^{d_m}\|_m=0\ \ \text{and}\
\ \lim_{n\to\infty}\max_{1\leq j<d_m}\|D^{\beta_n}\phi_n^j\|_m=0.
\end{equation}

It is easy to see that if we set $n_m=n$, then $r_m=\rho_n$,
$\nu_m=\beta_n$ and $f_m=\phi_n$. Thus formulae (\ref{e1}),
(\ref{e4}) and (\ref{e5}) imply that (a1--a4) for $k=m$ are
satisfied if we choose $n_m=n$ being large enough. This completes
the inductive construction of the sequence $\{n_k\}$ satisfying
(a1--a4).

Condition (a2) and formula (\ref{norm1}) imply that $\|r_k\|_m\leq
2^{-k}$ for any $k\geq m$. Hence the series
$\sum\limits_{k=1}^\infty r_k$ converges in $\H$ to some $f\in\H$.
Equivalently, $f$ is the limit in $\H$ of the sequence
$\{f_k\}_{k\in\N}$. In order to prove Theorem~\ref{main2}, it
suffices to demonstrate that $p\circ f\in H(D)$ for any non-constant
polynomial $p$. Since the set $H(D)$ is closed under multiplication
by non-zero scalars, it is enough to show that $p\circ f\in H(D)$ if
a polynomial $p$ has shape
$$
p(z)=z^d+\sum_{j=0}^{d-1} c_jz^j,\quad d\in\N.
$$
By Lemma~\ref{NtP}, the set $B=\{k\in\N:d_k=d\}$ is infinite and the
set $\{p_k:k\in B\}$ is dense in $\H$. Thus it is enough to prove
that
\begin{equation}\label{final}
p_k-D^{\nu_k}(p\circ f)\to 0\text{\ \ in $\H$ as $k\to\infty$, $k\in
B$,}
\end{equation}
where $\nu_k=n_k^2+(d_k-1)n_k$. Indeed, if it is the case, then
density of $\{p_k:k\in B\}$ in $\H$ implies density of
$\{D^{\nu_k}(p\circ f):k\in B\}$ in $\H$ and therefore
hypercyclicity of $p\circ f$ for $D$. It remains to prove
(\ref{final}). Let $m\in\N$ and $k\in B$ be such that $k\geq m$.
Since $d_k=d$ and $m\leq k$, condition (a4) implies
$$
\|p_k-D^{\nu_k}(f_k^d)\|_m\leq 2^{-k}\ \ \text{and}\ \
\|D^{\nu_k}(f_k^j)\|_m\leq 2^{-k}\ \ \text{for $1\leq j<d$}.
$$
Let $\nu\in\N$, $\nu\geq k$. Then $d=d_k\leq k\leq\nu$ and according
to (a3), we have
$$
\|D^{\nu_k}(f_{\nu+1}^j)-D^{\nu_k}(f_{\nu}^j)\|_m\leq 2^{-\nu-1} \ \
\text{for $1\leq j\leq d$}.
$$
Using the triangle inequality and the above two displays, we get
that
$$
\|p_k-D^{\nu_k}(f_n^d)\|_m\leq 2^{-k}+\sum_{\nu=k}^n
2^{-\nu-1}\leq2^{1-k}\ \ \text{and}\ \ \|D^{\nu_k}(f_n^j)\|_m\leq
2^{-k}+\sum_{\nu=k}^n 2^{-\nu-1}\leq2^{1-k}\ \ \text{for $1\leq
j<d$}
$$
for any $n\geq k$. Since $f_n\to f$ in $\H$ we have $f_n^j\to f^j$
in $\H$ as $n\to\infty$ for any $j\in\N$. Therefore passing to the
limit as $n\to\infty$ in the above display, we obtain
$$
\|p_k-D^{\nu_k}(f^d)\|_m\leq 2^{1-k}\ \ \text{and}\ \
\|D^{\nu_k}(f^j)\|_m\leq 2^{1-k}\ \ \text{for $1\leq j<d$, $k\in B$,
$k\geq m$}.
$$
Hence
$$
\|p_k-D^{\nu_k}(p\circ f)\|_m\leq 2^{1-k}\biggl(1+\sum_{1\leq
j<d}|c_j|\biggr)\ \ \text{for $k\in B$, $k\geq m$.}
$$
Passing to the limit as $k\to\infty$, we see that
$\|p_k-D^{\nu_k}(p\circ f)\|_m\to 0$ as $k\to \infty$, $k\in B$.
Since $m\in\N$ is arbitrary, (\ref{final}) is satisfied. The proof
of Theorem~\ref{main2} is now complete.

\section{Remarks}

The function $f$ constructed in the proof of Theorem~\ref{main2} has
moderate growth. Namely, it is easy to see that
$|f(z)|=O(e^{(1+\epsilon)|z|})$ as $|z|\to\infty$ for each
$\epsilon>0$. In particular, $f$ has finite exponential type. It is
worth noting that a function from $H(D)$ can not have exponential
type $<1$ \cite{mac}. It is also easy to verify that there can be no
common growth restriction for the functions from the space $L$ from
Theorem~\ref{main1}. We would like to discuss possible modifications
of Questions~\ref{ar1} and~\ref{ar2}. Note that $H(D)$ contains (up
to the zero function) no non-trivial ideals in $\H$. Indeed, let
$f\in\H\setminus\{0\}$. Then the function
$g(z)=\overline{f(\overline{z})}$ also belongs to $\H\setminus\{0\}$
and $(fg)(\R)\subseteq \R$. The latter inclusion implies that
$fg\notin H(D)$. Indeed the set of functions real on the real axis
is closed and nowhere dense in $\H$ and is preserved by $D$. Thus
the ideal generated by $f$ contains the non-zero function $fg$,
which is not hypercyclic for $D$. Finally we would like to raise the
following question.

\begin{question}\label{qs} Does $H(D)$ contain all non-constant
functions from a non-trivial closed subalgebra of $\H$?
Equivalently, does there exist $f\in\H$ such that $g\circ f\in H(D)$
for any non-constant $g\in\H$?
\end{question}

It seems likely that the answer to the above question is negative.
To prove this it would be sufficient for any $f\in\H$ to find a
non-constant $g\in\H$ and a bounded sequence $\{z_n\}_{n\in\Z_+}$ in
$\C$ such that the sequence $\{(g\circ f)^{(n)}(z_n)\}_{n\in\Z_+}$
is bounded. Indeed, boundedness of the last sequence would prevent
$g\circ f$ from being hypercyclic for $D$.

Leon and Montes \cite{alf2} have shown that if $T$ is a continuous
linear operator on a Banach space $X$ and $\sigma_e(T)$, being the
set of $\lambda\in\C$ such that $T-\lambda I$ is not Fredholm, does
not intersect the closed unit ball $\{z\in\C:|z|\leq 1\}$, then
there is no closed infinite dimensional subspaces $L\subset X$ such
that $L\setminus\{0\}\subset H(T)$. It is easy to see that
$\sigma_e(D)=\varnothing$. Thus the above result does not carry
through to operators on Fr\'echet spaces.

\bigskip

{\bf Acknowledgements.} \ The author is grateful to the referee for
helpful comments and numerous corrections.

\small\rm

\vskip1truecm

\scshape

\noindent Stanislav Shkarin

\noindent Queens's University Belfast

\noindent Department of Pure Mathematics

\noindent University road, Belfast, BT7 1NN, UK

\noindent E-mail address: \qquad {\tt s.shkarin@qub.ac.uk}

\end{document}